\documentclass[11pt,a4paper,oneside]{amsart}

\newcounter{commentcounter}

\usepackage{amsmath} 
\usepackage{amssymb} 
\usepackage{amsthm} 
\usepackage{stmaryrd} 
\usepackage[english]{babel} 
\usepackage[font=small,justification=centering]{caption} 
\usepackage[nodayofweek]{datetime}
\usepackage[shortlabels]{enumitem} 
\usepackage[T1]{fontenc} 
\usepackage[utf8]{inputenc} 
\usepackage{ifthen} 
\usepackage{mathabx} 
\usepackage{mathtools} 
\usepackage[dvipsnames]{xcolor} 
\usepackage[pdftex,  colorlinks=true]{hyperref} 
    \hypersetup{urlcolor=RoyalBlue, linkcolor=RoyalBlue,  citecolor=black}
\usepackage{setspace} 
\usepackage{tikz-cd} 
\usepackage{xfrac} 
\usepackage[capitalize]{cleveref} 
\usepackage{stmaryrd}
\usepackage{footnote}

\makeatletter
\@namedef{subjclassname@1991}{Mathematical subject classification 1991}
\@namedef{subjclassname@2000}{Mathematical subject classification 2000}
\@namedef{subjclassname@2010}{Mathematical subject classification 2010}
\@namedef{subjclassname@2020}{Mathematical subject classification 2020}
\makeatother

\newtheorem{thm}{Theorem}[section]
\newtheorem{lemma}[thm]{Lemma}
\newtheorem{corollary}[thm]{Corollary}
\newtheorem{prop}[thm]{Proposition}

\newtheorem{question}[thm]{Question}

\newtheorem{thmx}{Theorem}
\newtheorem{corx}[thmx]{Corollary}

\theoremstyle{definition}
\newtheorem{defn}[thm]{Definition}

\theoremstyle{plain}

    \newtheoremstyle{TheoremNum}
        {\topsep}{\topsep} 
        {\itshape} 
        {-0.25cm} 
        {\bfseries} 
        {.} 
        { }  
        {\thmname{#1}\thmnote{ \bfseries #3}}
    \theoremstyle{TheoremNum}

\newcommand*{\claimproofname}{My proof}

\def\Z{\mathbb{Z}}

\newcommand{\Gv}{\mathcal{G}_v}
\newcommand{\Gw}{\mathcal{G}_w}
\newcommand{\Ge}{\mathcal{G}_e}

\usepackage{tikz}
\usetikzlibrary{arrows,quotes}
\tikzstyle{blackNode}=[fill=black, draw=black, shape=circle]

\title[LERF property of F-by-Z's and deficiency 1 groups]{On subgroup separability of free-by-cyclic and deficiency 1 groups}
\author{Monika Kudlinska}
\date{September 2022}

\address{Mathematical Institute, Andrew Wiles Building, Observatory Quarter, University of Oxford, Oxford OX2 6GG, UK}
\email{monika.kudlinska@maths.ox.ac.uk}
\date{\today}


\begin{document}

\maketitle

\begin{abstract}
We show that a free-by-cyclic group with a polynomially growing monodromy is subgroup separable exactly when it is virtually $F_n \times \mathbb{Z}$. We also prove that random deficiency 1 groups are not subgroup separable with positive asymptotic probability.
\end{abstract}

\section{Introduction}

A group $G$ is said to be \emph{subgroup separable}, or \emph{locally extended residually finite (LERF),} if every finitely generated subgroup $H \leq G$ is the intersection of finite index subgroups of $G$. Subgroup separability initially gained prominence through its applications to low-dimensional topology and specifically 3-manifold theory, as it allows for certain immersions to be lifted to embeddings in finite index covers. It has since become useful in a much wider group theoretic setting, in particular in proving profinite rigidity results. For instance, Hughes--Kielak \cite{HughesKielak2022} showed that algebraic fibring is a profinite invariant of LERF groups.

By a classical result of M.~Hall \cite{Hall}, finitely generated free groups are known to be subgroup separable. More recently, D.~Wise \cite{Wise} showed that if $G$ is the fundamental group of a finite graph of finite rank free groups with cyclic edge groups, then $G$ is subgroup separable if and only if it is \emph{balanced}; that is, there does not exist a non-trivial element $g \in G$ such that $g^n$ is conjugate to $g^m$, for some $n \neq \pm m$. Any free-by-cyclic group is balanced, and furthermore, if it admits a linearly growing UPG monodromy then it can be realised as a mapping torus of a cyclic splitting of a finite-rank free group  \cite[Proposition 5.2.2]{NM}. It is therefore tempting to conjecture that such free-by-cyclic groups are subgroup separable. The aim of this paper is to show that this is almost never true.

\begin{thmx}\label{main}
Let $\Phi \in \mathrm{Out}(F_n)$ be a polynomially growing outer automorphism. Then $G = F_n \rtimes_{\Phi} \mathbb{Z}$ is subgroup separable if and only if $\Phi$ is periodic
\end{thmx}

Since the property of being LERF passes to subgroups, Theorem~\ref{main} combined with standard results on polynomial subgroups of free-by-cyclic groups (see e.g. \cite[Proposition~1.4]{Levitt}), implies the following corollary:
\begin{corx}\label{cor}
Let $\Phi \in \mathrm{Out}(F_n)$ and let $G = F_n \rtimes_{\Phi} \mathbb{Z}$. \begin{enumerate}
    \item If $\Phi$ acts periodically on every conjugacy class of elements in $F_n$ (equivalently, if $\Phi$ is a finite order outer automorphism) then $G$ is subgroup separable.
    \item  If there exists a conjugacy class $\bar{g}$ in $F_n$ which grows polynomially of order $d > 0$ under the action of $\Phi$ then $G$ is not subgroup separable.
\end{enumerate}
\end{corx}

Let $S$ be a compact surface with non-empty boundary and $[f]$ a pseudo-Anosov mapping class of $S$. The fundamental group of a finite-volume hyperbolic 3-manifold is subgroup separable (see e.g. \cite[Diagram~4]{AFW}). Hence, the fundamental group of the mapping torus $M_f$ of $f$ is a free-by-cyclic group which is LERF. The induced outer automorphism $f_{*}$ of $\pi_1(S) \simeq F_n$ acts periodically on the conjugacy classes of $F_n$ corresponding to the boundary components of $S$, and exponentially on the remaining conjugacy classes. This leads to the following natural question:

\begin{question}\label{qn}
Let $\Phi \in \mathrm{Out}(F_n)$ be an outer automorphism of $F_n$ such that for every conjugacy class of $F_n$, $\Phi$ acts periodically or exponentially, and there exists at least one conjugacy class with each type of growth. Suppose that $G = F_n \rtimes_{\Phi} \mathbb{Z}$ is LERF. Does it follow that $\Phi$ is geometric?
\end{question}

\cref{main} and \cref{qn} do not address the question of which free-by-cyclic groups with \emph{purely exponential} monodromy are subgroup separable. Combining classical results in the literature \cite{Brinkmann00, BestvinaHandel92}, it follows that mapping tori of purely exponential elements in $\mathrm{Out}(F_n)$ are exactly the Gromov hyperbolic free-by-cyclic groups, and thus we ask the following:

\begin{question}
    Which Gromov hyperbolic free-by-cyclic groups are subgroup separable?
\end{question}

Leary--Niblo--Wise construct examples of hyperbolic free-by-cyclic groups which are not subgroup separable, by realising such groups as ascending, non-descending HNN extensions of finitely generated free groups \cite{LNW}. It is interesting to note that whilst the failure of subgroup separability in the Leary--Niblo--Wise examples is due to the non-symmetric nature of the BNS invariant, free-by-cyclic groups with polynomially growing monodromies have symmetric BNS invariants \cite{CashenLevitt2016}.

It is a general fact that if a group $G$ has an integral character $\phi \colon G \to \Z$ which is contained in the BNS invariant, and $-\phi$ is not contained in the BNS invariant, then $G$ is not subgroup separable. We can leverage this fact to prove results about subgroup separability for random groups which admit deficiency 1 presentations.
\begin{thmx}\label{random}
Let $G$ be a random group of deficiency 1 with respect to the few-relator model (see Section~\ref{RandomSection}). Then, with positive asymptotic probability $G$ is not subgroup separable.
\end{thmx}
 Kielak--Kropholler--Wilkes show that a random few-relator deficiency 1 group is free-by-cyclic with positive asymptotic probability \cite{KKW}. Our methods for proving Theorem~\ref{random} imply that such a group is \emph{not} free-by-cyclic with positive probability, generalising a result of Dunfield--Thurston \cite{DT} who prove this for 2-generator 1-relator groups. 
\begin{thmx}\label{def1}
Let $G$ be a random group of deficiency 1 with respect to the few-relator model. Then $G$ is free-by-cyclic with asymptotic probability bounded away from 1.
\end{thmx}
A random 2-generator 1-relator group is virtually free-by-cyclic, almost surely \cite[Corollary~2.10]{KKW}. The analogous result in the deficiency 1 case is not known. 

\subsection*{Acknowledgments}
I would like to thank my supervisors Martin Bridson and Dawid Kielak for carefully reading earlier drafts of this paper, and for their continuing support and guidance; Naomi Andrew for helpful conversations on polynomially growing free-by-cyclic groups; finally, Sam Hughes and Nicholas Touikan for insightful discussions and correspondence which inspired this project. I'd also like to thank the anonymous referee for many valuable comments and suggestions.

This work was supported by an Engineering and Physical Sciences Research Council studentship (Project Reference 2422910).


\section{Background}

\subsection{Growth of free group automorphisms}

Let $F$ be a finite rank free group and fix a free set of generators $S$ of $F$. For any $g\in F$, we denote by $|g|$ the length of the reduced word representative of $g$. We write $|\bar{g}|$ to denote the minimal length of a cyclically reduced word representing a conjugate of $g$. 

An outer automorphism $\Phi \in \mathrm{Out}(F)$ acts on the set of conjugacy classes of elements in $F$. Given a conjugacy class $\bar{g}$ of an element $g\in F$, we say that \emph{$\bar{g}$ grows polynomially of degree $d$ under the iteration of $\Phi$,} if there exist constants $C_1, C_2 > 0$ such that for all $n\geq 1$, \[C_1 n^d \leq |\Phi^n(\bar{g})| \leq C_2 n^d.\] 

We say that $\Phi$ \emph{grows polynomially of degree $d$} if every conjugacy class of elements of $F$ grows polynomially of degree $\leq d$ under the iteration of $\Phi$, and there exists a conjugacy class which grows polynomially of degree exactly $d$.

Let $\Gamma$ be a graph. We will assume that every map $f \colon \Gamma \to \Gamma$ 
sends vertices to vertices, and edges to immersed non-trivial edge paths. Given an edge path $\gamma$ in $\Gamma$, we write $|\gamma|$ to denote the minimal simplicial length of an edge path in the homotopy class of $\gamma$, rel. endpoints.  We define polynomial growth of degree $d$ of an edge path $\gamma$ in $\Gamma$ under the iteration of $f$ analogously to the definition of the growth of a conjugacy class. Similarly for the growth of the map $f$.

If $\Phi$ is an UPG outer automorphism of $F$ with growth of order $d > 0$ then the growth of any improved relative train track representative of $\Phi$ is polynomial of order $d$ (see e.g. \cite[Lemma~2.3]{AHK}). Note however that it is possible for a linearly growing improved relative train track map to represent the identity outer automorphism.

We collect two results which will be useful in the sequel. Both can be found in the article of Levitt \cite[Theorem~1, Lemma~1.7]{Levitt}, though we note that they can be deduced without too much effort from the (improved) relative train track machinery.
\begin{lemma}\label{Levitt}
Let $\Phi \in \mathrm{Out}(F_n)$ be an outer automorphism which grows polynomially of degree $d$. Then, the following hold:
\begin{itemize}
    \item $d \leq n-1$;
    \item $d = 0$ if and only if $\Phi$ has finite order in $\mathrm{Out}(F_n)$.
\end{itemize}
\end{lemma}

\subsection{Group rings and the Bieri--Neumann--Strebel invariant}

Let $G$ be a group and $\phi \colon G \to \mathbb{Z}$ a homomorphism. The \emph{Novikov ring} $\widehat{\mathbb{Q}G}^{\phi}$ of $G$ with respect to $\phi$, is the set of all formal sums $x = \sum_{g\in G} {\lambda}_g g$ where $\lambda_g \in \mathbb{Q}$, such that for any $r \in \mathbb{R}$, the intersection $\mathrm{supp}(x) \cap \phi^{-1}((-\infty, r])$ is a finite set. Multiplication and addition in $\widehat{\mathbb{Q}G}^{\phi}$ are defined in the obvious way, so that the natural inclusion $\mathbb{Q}G \leq \widehat{\mathbb{Q}G}^{\phi}$ is an embedding of rings.

\begin{lemma}\label{units}
Let $G$ be group and $\phi \colon G \to \mathbb{Z}$ a homomorphism. Then, for every infinite-order element $g\in G$ and $\alpha \in \mathbb{Q}^{\times}$, $g- \alpha$ is a unit in $\widehat{\mathbb{Q}G}^{\phi}$ if and only if $\phi(g)\neq 0$.
\end{lemma}

\begin{proof}

Suppose $\phi(g) \neq 0$. If $\phi(g) > 0$ then the formal sum \[h = \alpha^{-1}\cdot\sum_{i = 0}^{\infty} (\alpha^{-1}g)^i\] is an element of $\widehat{\mathbb{Q}G}^{\phi}$, and $(\alpha - g)h = h(\alpha -g) = 1$. If $\phi(g) < 0$, then $\phi(g^{-1}) > 0$ and since $g$ is a unit in $\widehat{\mathbb{Q}G}^{\phi}$, it follows that $g-\alpha = \alpha g (\alpha^{-1}-g^{-1})$ is also a unit.

Suppose that $\phi(g) = 0$. For contradiction, assume that there exists some $h \in \widehat{\mathbb{Q}G}^{\phi}$ such that $(g-\alpha)h = 1$. Write $h = \sum_{k\in G} \lambda_k k$, where the coefficients $\lambda_k \in \mathbb{Q}$ are such that for any $r \in \mathbb{R}$, there are only finitely many elements $k \in G$ with $\phi(k) \leq r$ and $\lambda_k \neq 0$. Since $(g-\alpha)h = 1$, we have that for all $n > 0$, $\lambda_{g^n} = \alpha^{-n} \cdot \lambda_{1_G}$ and $\lambda_{g^{-n}} = \alpha^{n-1}(\alpha\cdot \lambda_{1_G} + 1)$. Hence $\lambda_{g^n} \neq 0$ for all $n > 0$, or $\lambda_{g^{-n}} \neq 0$ for all $n > 0$. However $\phi(g^{n}) = n \cdot \phi(g) = 0$ for all $n \in \mathbb{Z}$. Since $g \in G$ has infinite order, it follows that $\mathrm{supp}(h) \cap \phi^{-1}((-\infty, 0 ])$ is infinite. This is a contradiction.\end{proof}

The significance of the Novikov ring lies in its relation to the Bieri--Neumann--Strebel invariant of a group $G$. 

\begin{defn}\cite{BNS}
The \emph{Bieri--Neumann--Strebel invariant} (also known as the \emph{BNS invariant}) $\Sigma(G)$ of a group $G$, is the set of non-zero homomorphisms $\phi\colon G \to \mathbb{R}$ such that the monoid $\{g\in G \mid \phi(g) \geq 0\}$ is finitely generated.
\end{defn}

The original version of the following theorem, where the coefficient ring is equal to $\widehat{\mathbb{Z}G}^{\phi},$ is attributed to Sikorav and can be found in his PhD thesis \cite{Sikorav}. The proof of the result over $\mathbb{Q}$ is outlined in \cite[Theorem~3.11]{Kielak}.

\begin{thm}\label{Sikorav}
Let $G$ be a finitely generated group and $\phi \colon G \to \mathbb{Z}$ an epimorphism. Then $\phi$ is an element of the BNS invariant $\Sigma(G)$ of $G$ if and only if $H_1(G; \widehat{\mathbb{Q}G}^{\phi}) = 0$.
\end{thm}

The non-symmetric nature of the BNS invariant provides a useful criterion for detecting when a group is not subgroup separable. 

\begin{lemma}\label{separability}
If the set $\Sigma(G) \cap H^1(G; \mathbb{Z})$ of integral characters of $G$ which are contained in the BNS invariant is non-symmetric under the antipodal involution, then $G$ is not subgroup separable. 
\end{lemma}

\begin{proof}
Let $\phi \in \Sigma(G) \cap H^1(G; \mathbb{Z})$ be such that $ \phi \not \in -\Sigma(G)$. Proposition~4.1 in \cite{BNS} implies that there exists a finitely generated subgroup $A \leq G$ and an injective, non-surjective endomorphism $\theta \colon A \to A$, such that $G \simeq A \ast_{\theta}$. A standard argument (see e.g. \cite[Proposition 4]{LNW}) shows that if $G$ contains subgroups $B \leq A$ which are conjugate in $G$, then $B$ cannot be separated from any $g\in A \setminus B$ in any finite quotient of $G$. Hence $\mathrm{im}(\theta)$ is a non-separable subgroup of $G$.
\end{proof}

\subsection{Random groups}\label{RandomSection}

Let $k \in \mathbb{Z}$. A \emph{deficiency $k$ presentation} is a presentation of the form 
\[\langle x_1, \ldots, x_{m} \mid r_1, \ldots, r_n \rangle,\] where $m -n = k$, and $r_1, \ldots, r_n$ are non-empty reduced words in the alphabet $\{x_1^{\pm}, \ldots, x_m^{\pm}\}$. A group $G$ is said to be of \emph{deficiency $k$} if it admits a deficiency $k$ presentation and it does not admit a deficiency $k'$ presentation, for any $k' \geq k$. 

In this article, we will use the few-relator model for random groups. After fixing $n \geq 1 $ and $m \geq 1$, and for every $l \geq 1$, we write $\mathcal{R}_l$ to denote the set of group presentations of the form $\langle x_1, \ldots, x_m \mid r_1, \ldots, r_n \rangle$, where each $r_i$ is a cyclically reduced non-empty word in the alphabet $\{x_1^{\pm}, \ldots, x_m^{\pm}\}$ of length $\leq l$. For any given property $P$ of groups, we say that a presentation \emph{satisfies the property $P$} if the corresponding group satisfies it. The property $P$ is said to hold with \emph{asymptotic probability $p$}, for some $0 \leq p \leq 1$, if 
\[\frac{\#\{ \text{presentations in }\mathcal{R}_l \text{ which satisfy P}\}}{\#\mathcal{R}_l} \to p \text{ as }l \to \infty.\]
The property $P$ is said to hold with \emph{positive asymptotic probability} if 
\[\mathrm{lim \, inf}_{l \to \infty} \frac{\#\{ \text{presentations in }\mathcal{R}_l \text{ which satisfy P}\}}{\#\mathcal{R}_l} > 0. \] The probability is said to be \emph{bounded away from 1} if it holds with asymptotic probability $p <1$. Finally, we say that the property $P$ holds \emph{almost surely} if it holds with asymptotic probability $p = 1$.

A random presentation on $m$ generators and $n$ relators, with $m -n = k$, will correspond to a deficiency $k$ group, almost surely \cite{Wilkes19}. Hence, it makes sense to talk of a random deficiency $k$ group. 

\section{Polynomially growing automorphisms}

An outer automorphism $\Phi \in \mathrm{Out}(F_n)$ is said to be \emph{unipotent polynomially growing} (abbreviated to \emph{UPG}), if it is polynomially growing and its image in $\mathrm{GL}(n, \Z)$ is unipotent. It is a well-known fact that every polynomially growing $\Phi \in \mathrm{Out}(F_n)$ has a positive power $\Phi^k$ which is UPG.

The aim of this section is to prove that mapping tori of polynomially growing outer automorphisms are non-subgroup separable, unless the outer automorphism has finite order. We will reduce the problem to studying \emph{linearly growing UPG} outer automorphisms. The mapping tori of such automorphisms are analogous to fundamental groups of graph manifolds. More precisely, 

\begin{prop}\label{linear_splitting}
    Let $\Phi \in \mathrm{Out}(F_n)$ be a linearly growing UPG element and $G = F_n \rtimes_{\Phi} \Z$. There exists a simplicial $G$-tree $T$ with maximal $\Z^2$ edge stabilisers and maximal $F_m \times \Z$ vertex stabilisers, where $2 \leq m \leq n$. 
\end{prop}

The above proposition follows from the Parabolic Orbits Theorem in \cite{CohenLustig1995, CohenLustig1999} (see also \cite[Theorem~2.4.9]{AndrewMartino2022} and the discussion which follows it).

Inspired by the work of Niblo--Wise \cite{NW} on subgroup separability of graph manifolds, we will show show that the mapping torus of every linearly growing UPG element $\Phi \in \mathrm{Out}(F_n)$ contains a non-subgroup separable ``poison'' subgroup. Since subgroup separability is a property which passes to subgroups, this will be enough to conclude that $F_n \rtimes_{\Phi} \Z$ is not subgroup separable.

Our poison subgroup arises as the fundamental group of a link compliment given by the presentation
\[G_{NW} = \langle i,j, k,l \mid [i,j],[j,k], [k,l] \rangle.\]
Niblo--Wise proved that it is not subgroup separable in \cite[Theorem~1.2]{NW}, building on the work of Burns--Karrass--Solitar \cite{BKS}.

The proof of the following proposition is strongly inspired by the argument used to prove Lemma~4.1 and Theorem~2.1 in \cite{NW}.

\begin{prop}\label{linear_non_LERF}
    Let $\Phi \in \mathrm{Out}(F_n)$ be a linearly growing UPG element and $G = F_n \rtimes_{\Phi} \Z$. The group $G_{NW}$ embeds as a subgroup of $G$.
\end{prop}

\begin{proof}
    Let $T$ be the simplicial tree obtained in \cref{linear_splitting} and identify $G$ with the fundamental group of the quotient graph of groups, $G \simeq \pi_1(T/G, v)$. Note that since $\Phi$ has growth of order greater than zero, it must be the case that the vertex $v$ admits at least one incident edge $e$.

    Suppose first that $e = [v,w]$ has distinct endpoints $v \neq w$. Note that the stabilisers $\Gv$ and $\Gw$ of the vertices $v$ and $w$, respectively, have infinite cyclic centers. Let $t_v \in \mathcal{G}_v$ and $t_w \in \mathcal{G}_w$ be the generators of the centers of $\mathcal{G}_v$ and $\mathcal{G}_w$, respectively. Since $\mathcal{G}_v$ is a direct product of a free group with an infinite cyclic group, and $\mathcal{G}_e \simeq \mathbb{Z}^2$ is a maximal $\mathbb{Z}^2$-subgroup of $\mathcal{G}_v$, it follows that $\langle t_v  \rangle \leq \Ge$ is a maximal cyclic subgroup. by the same argument, $ \langle t_w \rangle \leq \Ge$ is a maximal cyclic subgroup. 
    
    Suppose that $t_v \in Z(\Gw)$. Then $t_v = t_w^{\epsilon}$ where $\epsilon \in \{\pm 1\}$. It follows that the subgroup of $G$ generated by $\Gv$ and $\Gw$ has infinite cyclic center. Since every subgroup of a free-by-cyclic group is free-by-cyclic, where the free factor possibly has infinite rank, and since $\langle \Gv, \Gw \rangle$ is finitely generated, we deduce that it is of the form $F \times \mathbb{Z}$ where $F$ is finite-rank free. But this contradicts the maximality of $\Gv$. Hence $t_v$ is not contained in the center of $\Gw$, and so there exists an element $g_w \in 
    \Gw$ which does not commute with $t_v$. Similarly, there exists an element $g_v \in \Gv$ which does not commute with $t_w$.

    Let $K = \langle g_v, t_v, t_w, g_w \rangle \leq G$. We have that $[g_v, t_v] = [t_v, t_w] = [t_w, g_w] = 1$, and so there exists an epimorphism 
    \[\begin{split}
       G_{NW} &\to K\\
        i \mapsto g_v, j\mapsto t_v, k &\mapsto t_w, l \mapsto g_w. 
    \end{split}
    \]
    Normal forms show that the epimorphism is injective.

    Suppose now that the endpoints of the edge $e$ coincide. Let $s \in G$ be the element of $G$ corresponding to the loop $e$. Then, the subgroup of $G$ generated by the stabiliser $\Gv$ of $v$ and its conjugate $\Gv^s := s^{-1}\Gv s$, is isomorphic to $\langle \Gv, \Gv^s \rangle \simeq \Gv \ast_{\Ge} \widebar{\Gv}$, where $\widebar{\Gv} \simeq \Gv^s$ denotes a copy of $\Gv$. Observe now that the same argument as above proves that $G_{NW}$ embeds into $\langle \Gv, \Gv^s \rangle$, and thus into $G$. 
\end{proof}

The following result will be used to reduce to the case of linearly growing outer automorphisms. It exists in the literature in various forms (see e.g. \cite{Macura2002}, \cite{Hagen2019}). The proof of the exact statement below can be found in \cite[Proposition~2.5]{AHK}.

\begin{prop}\label{splitting}
    Let $\Phi \in \mathrm{Out}(F_n)$ be UPG with growth of order $d \geq 2$. Then $G = F_n \rtimes_{\Phi} \Z$ splits as the fundamental group of a graph of groups with vertex groups isomorphic to mapping tori $F_k \rtimes_{\Psi} \Z$ with $\Psi \in \mathrm{Out}(F_k)$ a UPG element with growth of order at most $d-1$; moreover, there exists a vertex group with polynomially growing monodromy of order exactly $d-1$.
\end{prop}

We are now ready to state and prove our main theorem:

\begin{thm}[Theorem~A] Let $\Phi \in \mathrm{Out}(F_n)$ be a polynomially growing outer automorphism and $G = F_n \rtimes_{\Phi} \mathbb{Z}$. Then, the following are equivalent:
\begin{enumerate}
    \item The outer automorphism $\Phi$ has growth of order $d = 0$;
    \item $G$ is virtually a direct product $F_n \times \Z$;
    \item $G$ is subgroup separable;
    \item $G$ does not contain $G_{NW}$.
\end{enumerate}
\end{thm}

\begin{proof}

The equivalence of $(1)$ and $(2)$ follows from \cref{Levitt}. The implication $(2) \Rightarrow (3)$ follows from the well-known fact that direct products of the form $F_n \times \Z$ are subgroup separable (see e.g. \cite{AllenbyGregorac1973}), and the fact that subgroup separability is a commensurability invariant. Niblo--Wise show that $G_{NW}$ is not subgroup separable in \cite[Thoerem~1.2]{NW}, and thus any subgroup which contains $G_{NW}$ is not subgroup separable, which gives the implication $(3) \Rightarrow (4)$.

It remains to show $(4) \Rightarrow (1)$. Suppose that $\Phi$ has growth of order $d > 0$. We recall that for every outer automorphism $\Phi \in \mathrm{Out}(F_n)$ there exists some integer $k > 0$ such that $\Phi^k$ is UPG. The element $\Phi^k$ has the same polynomial growth rate as $\Phi$, and $G' := F_n \rtimes_{\Phi^k} \Z$ is isomorphic to a finite index subgroup of $G = F_n \rtimes_{\Phi} \Z$. Since subgroup separability is a commensurability invariant, it thus follows that there is no loss of generality in assuming that $\Phi$ is UPG.

The proof now follows by induction on the degree of growth. The base case $d = 1$ is exactly \cref{linear_non_LERF}. For the inductive step, apply \cref{splitting} to deduce that if $\Phi$ has growth of order $d \geq 2$, then it splits as a graph of groups where some vertex group is a mapping torus of an UPG outer automorphism with growth of order $d-1$.
\end{proof}

\section{Generic behaviour of deficiency 1 groups}

In this section, we study the BNS invariant of a random deficiency 1 group. To that end, Lemma~\ref{sigma invariant} characterises the maps $\phi \colon G \to \mathbb{Z}$ which are \emph{not} contained in the BNS invariant of $G$, in terms of the minima of $\phi$ evaluated at the suffixes of the relators. This approach is similar in flavour to that of Brown's algorithm \cite{Brown87}, a classical tool used to calculate the BNS invariant of 2-generator 1-relator groups. However, as we are (in general) no longer in the realm of 1-relator groups, the characterisation that we obtain is less clean than that in \cite{Brown87}, and the methods used to prove it are completely different. 

Let $R$ be a ring and $t$ a formal symbol. We write $R(\!(t)\!)$ to denote the set of Laurent power series over $R$ with a single variable $t$,
\[R(\!( t )\!) = \left\{ \sum_{i \geq k} a_it^{i} \mid a_i \in R, k \in \mathbb{Z}\right\}. \]
Let $\alpha$ be an automorphism of $R$. The ring of \emph{twisted Laurent series} is the set $R(\!(t)\!)$, with the obvious summation and multiplication defined by linearly extending 
\[r_1t^{n_1} \cdot r_2t^{n_2} := r_1 \alpha^{n_1}(r_2) t^{n_1 + n_2},\] for all $r_1, r_2 \in R$ and $n_1, n_2 \in \mathbb{Z}$. The \emph{$t$-order} of a Laurent series $f \in R(\!(t)\!)$, denoted $\mathrm{ord}_t(f)$, is the lowest power of $t$ with a non-zero coefficient in the expansion of $f$. We define $\mathrm{ord}_t(0) = \infty.$

Let $G$ be a group and $\phi \colon G \to \mathbb{Z}$ a homomorphism. Let $t \in G$ be an element such that $\phi(t)$ generates $\mathbb{Z}$. Let $K = \mathrm{ker}(\phi)$ and let $\mathbb{Q}K (\!(t)\!)$ denote the ring of twisted Laurent series, where the twisting automorphism $\alpha$ is obtained by extending the automorphism of $K$ induced by the conjugation action of $t$ on $K$ in $G$. Then there is a natural identification $\widehat{\mathbb{Q}G}^{\phi} \simeq \mathbb{Q}K(\!(t)\!)$. Given a subset $S \subseteq \mathbb{Z}$, we say that $x \in \widehat{\mathbb{Q}G}^{\phi}$ is \emph{supported over $S$} if $x = \sum_{i \in S}a_it^{i}$, for some $a_i \in \mathbb{Q}K$.

\begin{lemma}\label{invertibility}
Let $G$ be a group such that the group ring $\mathbb{Q}G$ has no zero-divisors. Let $B$ and $P$ be $n \times n$ matrices over $\mathbb{Q}G$. Suppose that $B = \mathrm{diag}(k_1t^{\rho_1}, \ldots, k_nt^{\rho_n})$, where $k_i \in \mathbb{Q}K$ and $\rho_i \in \mathbb{Z}$ for every $1 \leq i \leq n$. Assume that $k_i \in K$ for $i > 1$ and that $k_1$ is not a unit in $\widehat{\mathbb{Q}G}^{\phi}$. Suppose that all the elements in the $i^{th}$ row of $P$ are supported over $\mathbb{Z} \cap [\rho_{i} + 1, \infty)$. Then the matrix $A = B+ P$ is not invertible over $\widehat{\mathbb{Q}G}^{\phi}$.
\end{lemma}
\begin{proof}
Since $t^{\rho_1}$ and $k_it^{\rho_i}$ for $i > 1$ are units in $\mathbb{Q}G$, the matrix \[M = \mathrm{diag}(t^{\rho_1}, k_2t^{\rho_2},  \ldots, k_nt^{\rho_n})\] is an invertible matrix over $\widehat{\mathbb{Q}G}^{\phi}$. Hence $A$ is invertible if and only if $A' = M^{-1}A$ is invertible. The diagonal elements of $A'$ other than the element in the first row are of the form $1 + p_{ii}$, for some $p_{ii} \in \widehat{\mathbb{Q}G}^{\phi}$ supported over a positive subset of the integers. Such elements are invertible over $\widehat{\mathbb{Q}G}^{\phi}$ and the inverse $(1+ p_{ii})^{-1}$ is an element supported over non-negative integers. Hence by applying elementary column operations over $\widehat{\mathbb{Q}G}^{\phi}$, we may transform $A'$ into an upper triangular matrix $A''$ where the first element on the diagonal is given by $k_1 + p'_{11}$, with $k_1 \in \mathbb{Q}K$ a non-unit, and $p'_{11} \in \widehat{\mathbb{Q}G}^{\phi}$, an element supported over the positive integers.  Since elementary column operations are invertible, again $A''$ is invertible if and only if $A'$ is invertible. 

Suppose now that $A''$ is invertible over $\widehat{\mathbb{Q}G}^{\phi}$ and let $C = (c_{ij})$ be the inverse. Then $c_{11}(k_1 + p'_{11}) = 1$. Since $\mathbb{Q}G$ does not have non-trivial zero-divisors, neither does $\widehat{\mathbb{Q}G}^{\phi}$. Hence for any elements $p,q \in \widehat{\mathbb{Q}G}^{\phi}$ $\mathrm{deg}_t(pq) = \mathrm{deg}_t(p) + \mathrm{deg}_t(q)$. Suppose that $\mathrm{ord}_t(c_{11}p_{11}') > 0$. Then $\mathrm{ord}_t(c_{11}k_1) = \mathrm{ord}_t(1 - c_{11}p_{11}') = 0$. Hence \[ 0 = \mathrm{ord}_t(c_{11} k_1) = \mathrm{ord}_t(c_{11}) + \mathrm{ord}_t(k_1) = \mathrm{ord}_t(c_{11}).\] Let $d \in \mathbb{Q}K$ be the coefficient of the $t^0$ term in $c_{11}$. Note that $d \neq 0$. Then $dk_{1} = 1$ and thus $k_1$ is a unit. Hence $\mathrm{ord}_t(c_{11}p_{11}') \leq 0$.

Suppose that $\mathrm{ord}_t(c_{11}p_{11}') < 0$. Then \[\mathrm{ord}_t(c_{11}k_1) = \mathrm{ord}_t(1 - c_{11}p_{11}') = \mathrm{ord}_t(c_{11}p_{11}').\]
Hence $\mathrm{ord}_t(c_{11}k_1) = \mathrm{ord}_t(c_{11}p_{11}')$. Thus \[0 = \mathrm{ord}_t(k_1) = \mathrm{ord}_t(p_{11}') >0.\] 
Hence, it must be the case that $\mathrm{ord}_t(c_{11}p_{11}') = 0$. But then $\mathrm{ord}_t(c_{11}) < 0$ and thus $\mathrm{ord}_t(c_{11}k_1) < 0$. But then $\mathrm{ord}_t(1 - c_{11}p_{11}') < 0$, which is impossible since $\mathrm{ord}_t(c_{11}p_{11}') = 0$. In all cases we get a contradiction, and thus $A''$ is not invertible.\end{proof}

Given a cyclically reduced word $w = w_1\cdots w_m$ in the alphabet $\{x_1^{\pm}, \ldots, x_{n+1}^{\pm}\}$ and $k \leq |w|$, we let $[w]_k = w_1 \ldots w_k$ be the prefix of $w$ of length $k$. Let $C_w$ denote the cyclic graph of length $|w|$, with a marked vertex $\ast$ and labelled edges, such that consecutive edges of $C_w$, starting at the vertex $\ast$ and moving in the clockwise direction, spell out the word $w$. Assign labels to vertices of $C_w$ so that the vertex $v$ is labelled by the word which is spelled out by the embedded path joining $\ast$ to $v$, in the clockwise direction. Let $\phi \colon F(x_1, \ldots, x_{n+1}) \to \mathbb{Z}$ be a homomorphism. There is an induced map $\tilde{\phi} \colon C_w \to \mathbb{R}$ defined by linearly extending the map $\phi$ from the labels of the vertices to the whole graph. We define the \emph{lower section} of $w$ to be the preimage 
\[L_{\phi}(w) = \tilde{\phi}^{-1}(\mathrm{min}\{\tilde{\phi}(p) \mid p \in C_r\}).\]

Let $(r_1, \ldots, r_n)$ be a collection of cyclically reduced words in the alphabet $\{x_1^{\pm}, \ldots, x_{n+1}^{\pm}\}$. Let $\phi \colon F(x_1, \ldots, x_{n+1}) \to \mathbb{Z}$ be a homomorphism. The tuple $((r_1, \ldots, r_n), \phi)$ is said to satisfy the \emph{unique minimum condition} if, after possible re-ordering, the following conditions are satisfied.
\begin{enumerate}
    \item We have that $\phi(x_i) \geq 0$ for each $i \leq n$ and $\phi(x_{n+1}) <0$. 
    \item The homomorphism $\phi$ vanishes on each $r_i$.
    \item The lower section $L_{\phi}(r_i)$ consists of exactly one of the following:
\begin{itemize}
    \item A single vertex such that one of the adjacent edges is labelled by $x_i^{\pm}$ and the other is labelled by $x_{n+1}^{\pm}$.
    \item A single edge labelled by $x_i^{\pm}$ such that the adjacent edges are labelled by $x_{n+1}^{\pm}$.
\end{itemize}
\end{enumerate}
The tuple $((r_1, \ldots, r_n), \phi)$ satisfies the \emph{repeated minimum condition} if it satisfies the unique minimal condition, except at a single relator $r_j$, for some $1 \leq j \leq m$, where $L_{\phi}(r_j)$ consists of two occurrences of a vertex, or two occurrences of an edge as in the unique minimum condition. In that case, we call $r_j$ the \emph{relator with a repeated minimum}.

Let $G$ be a group given by the deficiency 1 presentation \[G = \langle x_1, \ldots, x_{n + 1} \mid r_1, \ldots, r_n \rangle.\] Let $\phi \colon G \to \mathbb{Z}$ be a homomorphism with kernel $K$, and $t \in G$ an element such that $\phi(t)$ generates $\mathbb{Z}$. 

\begin{lemma}\label{matrix}
Suppose that $((r_1, \ldots, r_n), \phi)$ satisfies the repeated minimum condition, where $r_1$ is the relator with a repeated minimum. Then for each $i\leq n$, there exists some integer $P_i \in \mathbb{Z}$, and for every $j \leq n$ and $k \geq P_i$, there exist elements $u_{ij,k} \in \mathbb{Q}K$, such that the Fox derivatives of $r_i$ are of the form 
\[\frac{\partial r_i}{\partial x_j} = \sum_{k \geq P_i} u_{ij,k}t^{k},\]
such that for any $i \neq j$, the element $u_{ij, P_i} = 0$, and $u_{ii, P_i} \in K$ for $i \neq 1$, and $u_{11, P_1}$ is a non-unit in $\widehat{\mathbb{Q}G}^{\phi}$.
\end{lemma}

\begin{proof}

For every relator $r_i$ and generator $x_j$, the partial derivative $\frac{\partial r_i}{\partial x_j}$ is the sum of prefixes of $r_i$ of the form $ux_j^{-1}$ and $v$, where $v$ immediately precedes an instance of $x_j$ in $r_i$. For each $i$, let $P_i = \tilde{\phi}(L_{\phi}(r_i)) \in \mathbb{Z}$. Hence, for every summand $u$ of $\frac{\partial r_i}{\partial x_j} \in \mathbb{Z}G$, we have that $\phi(u) \geq P_i$ and $\phi(u) = P_i$ if and only if $u$ is the label of a vertex of $C_{r_i}$ contained in $L_{\phi}(r_i)$. Any such vertex has adjacent edges labelled by $x_i^{\pm}$ and $x_{n+1}^{\pm}$. In particular, either the prefix $u$ has $x_i^{\pm}$ as its last letter and is followed by $x_{n+1}^{\pm}$ in $r_i$, or the same holds but with the roles of $x_i$ and $x_{n+1}$ reversed. This implies that for every summand $u$ of $\frac{\partial r_i}{\partial x_j}$, if $i \neq j$ then $\phi(u) > P_i$.

Now suppose that $i > 1$. Let $\mathcal{A} = \{u_{\alpha}\}$ be the collection of summands of $\frac{\partial r_i}{\partial x_i}$ such that $\phi(u_{\alpha}) = P_i$. Each element of $\mathcal{A}$ must be the label of a vertex in $L_{\phi}(r_i)$.  Suppose that $L_{\phi}(r_i)$ is a single vertex with label $u$. Since each $\phi(x_i) \geq 0$,  either $u$ is followed by $x_i$ in $r_i$, or the final letter of $u$ is $x_i^{-1}$. In either case, $u \in \mathcal{\mathcal{A}}$ and thus $\mathcal{A}$ contains exactly one element. Suppose instead that $L_{\phi}(r_i)$ consists of two vertices $u$ and $ux_i^{\pm}$. Exactly one of these words is a summand of $\frac{\partial r_i}{\partial x_i}$, depending on whether we choose $x_i$ or $x_i^{-1}$. Hence, it follows in this case also that $\mathcal{A}$ contains exactly one element, and this element can be expressed as $kt^{P_i}$, for some $k \in K$.

Finally we consider $\frac{\partial r_1}{\partial x_1}$. Defining $\mathcal{A}$ as above, $\mathcal{A}$ has exactly two elements given by the reduced words $u$ and $uv$, where $\phi(u) = P_i$ and $\phi(v) = 0$, where $u$ is the label of the path joining the marked vertex $\ast$ to the first minimum vertex which is a summand of $\frac{\partial r_1}{\partial x_1}$, and $v$ is the label of the path joining the two minima. Then $u = kt^{P_i}$ and $uv = kv't^{P_i}$, for some $k, v' \in K$. Note that the element $1+v' \in \mathbb{Z}G$ is not invertible over $\widehat{\mathbb{Q}G}^{\phi}$ by Lemma~\ref{units}, and thus $k(1-v')$ is not a unit in $\widehat{\mathbb{Q}G}^{\phi}$.\end{proof}

\begin{lemma}\label{sigma invariant}
Let $G$ be a group given by the deficiency 1 presentation
\[G = \langle x_1, \ldots, x_{n+1} \mid r_1, \ldots, r_n \rangle.\] Suppose that $\mathbb{Q}G$ has no non-trivial zero-divisors. Let $\phi \colon G \to \mathbb{Z}$ be a homomorphism. \begin{enumerate}
    \item If $((r_1, \ldots, r_n),\phi)$ satisfies the unique minimum condition then $\phi \in \Sigma(G)$.
    \item If $((r_1, \ldots, r_n),\phi)$ satisfies the repeated minimum condition then $\phi \not \in \Sigma(G)$.
\end{enumerate} 
\end{lemma}

\begin{proof}
The first statement follows from \cite[Theorem~3.4]{KKW}.

For the second statement, by Theorem~\ref{Sikorav} it suffices to show that $H_1(G; \widehat{\mathbb{Q}G}^{\phi})$ is non-trivial whenever $((r_1, \ldots, r_n), \phi)$ satisfies the repeated minimum condition. To that end, consider the chain complex of $\mathbb{Q}G$-modules
\begin{equation}\label{chain complex}C_2 \xrightarrow{\partial_2} C_1 \xrightarrow{\partial_1} C_0.\end{equation}
Here, the $\mathbb{Q}G$-module $C_2$ is the free module of rank $n$ with an ordered basis identified with the relators $(r_1, \ldots, r_n)$. The $\mathbb{Q}G$-module $C_1$ is the free module of rank $n+1$ with an ordered basis identified with the generators $(x_1, \ldots, x_{n+1})$ and $C_0 = \mathbb{Q}G$. The boundary map $\partial_1$ is given by the column vector with entries $x_i - 1$, for $1 \leq i \leq n+1$, and the boundary map $\partial_2$ is the matrix $A$ of Fox derivatives $\left(\frac{\partial r_i}{\partial x_j} \right)$. After possible re-ordering, we may assume that $r_1$ is the relator with the repeated minimum. We tensor the chain complex \eqref{chain complex} with $\widehat{\mathbb{Q}G}^{\phi}$ and let $(e_1, \ldots, e_{n+1})$ be the resulting free generating set of $C_1 \otimes \widehat{\mathbb{Q}G}^{\phi}$. We write $A'$ to denote the matrix obtained from $A$ by restricting the image of the boundary map to the subspace spanned by  $\{e_1, \ldots, e_n\}$. We claim that
\[H_1(G, \widehat{\mathbb{Q}G}^{\phi}) = \mathrm{coker}(A').\] 
Since $\mathbb{Q}G$ has no non-trivial zero-divisors, the element $x_{n+1}$ has infinite order in $G$. By the definition of the repeated minimum condition, $\phi(x_{n+1}) \neq 0$ and thus by Lemma~\ref{units}, the element $x_{n+1} - 1$ is invertible over $\widehat{\mathbb{Q}G}^{\phi}$. Let us define a map 
\[\begin{split} C_1 \otimes \widehat{\mathbb{Q}G}^{\phi} &\to \mathrm{ker}(\partial_1 \otimes \mathrm{id}) \\
\sum_{i}^{n+1} \lambda_i e_i &\mapsto \sum_{i=1}^n \lambda_i e_i + \lambda_{n+1}'e_{n+1},
\end{split}
\] where $\lambda_{n+1}' = - \sum_{i=1}^{n}\lambda_i(x_i - 1)(x_{n+1} - 1)^{-1}$. This map is clearly onto and every element of $\mathrm{im}(A')$ is sent to $\mathrm{im}(A)$. This proves the claim.

Combining Lemma~\ref{matrix} with Lemma~\ref{invertibility} shows that $A'$ is non-invertible over $\widehat{\mathbb{Q}G}^{\phi}$. Thus $H_1(G, \widehat{\mathbb{Q}G}^{\phi}) \neq 0$. \end{proof}

We are now ready to prove the key lemma of the section, inspired by \cite[Proposition~5.1]{KKW}.

\begin{lemma}\label{min_max}
Let $G$ be a random group of deficiency 1. Then, with positive probability, there exists a character $\phi \colon G \to \mathbb{Z}$ such that $\phi$ satisfies the unique minimum condition and $-\phi$ satisfies the repeated minimum condition. 
\end{lemma}

\begin{proof}
For each positive integer $l$, let $\mathcal{R}_l$ denote the set of $n$-tuples $(r_1, \ldots, r_n)$ of cyclically reduced words in the alphabet $\{x_1^{\pm}, \ldots, x_{n+1}^{\pm}\}$ of positive length $\leq l$. We define $\mathcal{R}'_l$ to be the subset of $n$-tuples $(r_1, \ldots, r_n)$ in $\mathcal{R}_l$, such that the group $G = \langle x_1 \ldots, x_{n+1} \mid r_1 , \ldots, r_n \rangle$ has first Betti number equal to 1.
We let $\mathcal{T}$ denote the set of all deficiency 1 presentations such that the resulting group admits a homomorphism to $\mathbb{Z}$ which satisfies the hypotheses. To prove the lemma, it suffices to construct an injective map $f \colon \mathcal{R}'_l \to \mathcal{T} \cap \mathcal{R}'_{l+12}$. Then, 
\[ \frac{|\mathcal{T} \cap \mathcal{R}'_{l+12}|}{|\mathcal{R}'_{l+12}|} \geq \frac{|\mathcal{R}'_l|}{|\mathcal{R}'_{l+12}|} > \varepsilon, \]
where $\varepsilon > 0$ depends only on $n$. The result follows by noting that $|\mathcal{R}'_l|/|\mathcal{R}_l| \to 1$ as $l\to \infty.$

Now for each $n$-tuple $(r_1, \ldots, r_n) \in \mathcal{R}'_l$, there exists a non-trivial map $\phi \colon F(x_1, \ldots, x_{n+1}) \to \mathbb{Z}$ such that $\phi(r_i) = 0$ for every $i$. After possibly re-ordering and inverting the generators $x_i$, we can assume that $\phi(x_{i}) \geq 0$ for all $i$, and $\phi(x_{n+1}) < 0$.
Our strategy for defining $f$ is to alter the $n$-tuple $(r_1, \ldots, r_n)$ so that $\phi$ still vanishes on each element of the tuple, and such that it has a unique minimum and repeated maxima. We do this by inserting commutators. We follow the convention where $[x,y] = xyx^{-1}y^{-1}$.

For each relator $r_i$, form a new relator $r_i'$ by inserting a commutator $[x_{n+1}, x_i^{\epsilon}]$ at the first $\phi$-minimal vertex along $C_{r_i}$, where $\epsilon = -1$ if $\phi(x_i) > 0$ and $\epsilon = 1$ otherwise. Now for each $i > 1$, form a new relator ${r_i}''$ by inserting the commutator $[x_{n+1}^{-1}, x_i^{-\epsilon}]$ at the first $\phi$-maximal vertex along $C_{r_i'}$. Form $r_1''$ by inserting the square $[x_{n+1}^{-1}, x_1^{-\epsilon}]^2$ of the commutator at the first $\phi$-maximal vertex along $C_{r_1'}$. The lower section $L_{\phi}(r''_i)$ of each $r_i''$ consists of a single vertex or an edge labelled by the element $x_i$. The upper section $U_{\phi}(r''_{1})$ of $r_1''$ consists of two vertices or two edges labelled by $x_1$, and for $i > 1$ the upper section $U_{\phi}(r''_{i})$ of $r_i''$ consists of a single vertex or edge labelled by $x_i$. Hence $((r_1'', \ldots, r_n''), \phi)$ satisfies the unique minimum condition and $((r_1'', \ldots, r_n''), -\phi)$ satisfies the repeated minimum condition. The map $f$ is injective since there exists a left inverse $g \colon \mathrm{im}(f) \to \mathcal{R}_l$  of $f$ which acts by removing the commutators at the $\phi$-minimal and $\phi$-maximal vertices or edges of the $r_i''$. This completes the proof. \end{proof}

\begin{thm}\label{non-symmetricity}
Let $G$ be a random group of deficiency 1. Then with positive asymptotic probability, $\Sigma(G) \cap H^1(G; \mathbb{Z})$ is non-symmetric. 
\end{thm}

\begin{proof}
A random deficiency 1 presentation satisfies the $C''(\frac{1}{6})$ condition, almost surely \cite{Gromov1993}. Combining the work of Wise \cite{WiseCubes} and Agol \cite{Agol13}, it follows that such a group is virtually special and thus satisfies the Atiyah conjecture by \cite{Schreve14}. Hence $\mathbb{Q}G$ has no non-trivial zero-divisors. By Lemma~\ref{min_max}, a random deficiency 1 group admits a character $\phi$ such that $\phi$ satisfies the unique minimum condition and $-\phi$ satisfies the repeated minimum condition, with positive asymptotic probability. Thus by Lemma~\ref{sigma invariant}, $\phi \in \Sigma(G)$ and $-\phi \not \in \Sigma(G)$.\end{proof}

\begin{corollary}[Theorem~\ref{random}]
Let $G$ be a random group of deficiency 1. Then with positive asymptotic probability, $G$ is not subgroup separable.
\end{corollary}

\begin{proof}
Combine Theorem~\ref{non-symmetricity} with Lemma~\ref{separability}.\end{proof}

\begin{corollary}[Theorem~\ref{def1}]
Let $G$ be a random group of deficiency 1. Then $G$ is free-by-cyclic with asymptotic probability that is positive and bounded away from 1.
\end{corollary}

\begin{proof}
A random deficiency 1 group has first Betti number $b_1(G)$ equal to 1, almost surely. Hence $\mathrm{Hom}(G, \mathbb{Z}) \simeq \mathbb{Z} \simeq \langle \phi \rangle$. By Theorem~\ref{non-symmetricity}, $\Sigma(G) \cap H^1(G; \mathbb{Z})$ is non-empty and non-symmetric, with positive asymptotic probability. Hence $\Sigma(G) = \{\lambda \phi \mid \lambda \in \mathbb{R}_{>0}\}$ or $\Sigma(G) = \{\lambda \phi \mid \lambda \in \mathbb{R}_{<0}\}.$ In particular $\Sigma(G) \cap - \Sigma(G) = \emptyset$ and thus $G$ does not fibre algebraically. Hence the asymptotic probability that a random deficiency 1 group is free-by-cyclic is bounded away from 1. The fact that it is greater than 0 follows from \cite[Theorem~A]{KKW}.\end{proof}

\bibliographystyle{alpha}
\bibliography{refs.bib}

\begin{thebibliography}{KKW22}

\bibitem[AFW15]{AFW}
Matthias Aschenbrenner, Stefan Friedl, and Henry Wilton.
\newblock {\em 3-manifold groups}.
\newblock EMS Series of Lectures in Mathematics. European Mathematical Society
  (EMS), Z\"{u}rich, 2015.

\bibitem[AG73]{AllenbyGregorac1973}
R.~B. J.~T. Allenby and R.~J. Gregorac.
\newblock On locally extended residually finite groups.
\newblock In {\em Conference on {G}roup {T}heory ({U}niv.
  {W}isconsin-{P}arkside, {K}enosha, {W}is., 1972)}, Lecture Notes in Math.,
  Vol. 319, pages 9--17. Springer, Berlin, 1973.

\bibitem[Ago13]{Agol13}
Ian Agol.
\newblock The virtual {H}aken conjecture.
\newblock {\em Doc. Math.}, 18:1045--1087, 2013.
\newblock With an appendix by Agol, Daniel Groves, and Jason Manning.

\bibitem[AHK22]{AHK}
Naomi Andrew, Sam Hughes, and Monika Kudlinska.
\newblock Torsion homology growth of polynomially growing free-by-cyclic
  groups, 2022.
\newblock arXiv:2211.04389 [math.GR].

\bibitem[AM22a]{NM}
Naomi Andrew and Armando Martino.
\newblock Free-by-cyclic groups, automorphisms and actions on nearly canonical
  trees.
\newblock {\em J. Algebra}, 604:451--495, 2022.

\bibitem[AM22b]{AndrewMartino2022}
Naomi Andrew and Armando Martino.
\newblock Free-by-cyclic groups, automorphisms and actions on nearly canonical
  trees.
\newblock {\em J. Algebra}, 604:451--495, 2022.

\bibitem[BH92]{BestvinaHandel92}
Mladen Bestvina and Michael Handel.
\newblock Train tracks and automorphisms of free groups.
\newblock {\em Ann. of Math. (2)}, 135(1):1--51, 1992.

\bibitem[BKS87]{BKS}
R.~G. Burns, A.~Karrass, and D.~Solitar.
\newblock A note on groups with separable finitely generated subgroups.
\newblock {\em Bull. Austral. Math. Soc.}, 36(1):153--160, 1987.

\bibitem[BNS87]{BNS}
Robert Bieri, Walter~D. Neumann, and Ralph Strebel.
\newblock A geometric invariant of discrete groups.
\newblock {\em Invent. Math.}, 90(3):451--477, 1987.

\bibitem[Bri00]{Brinkmann00}
P.~Brinkmann.
\newblock Hyperbolic automorphisms of free groups.
\newblock {\em Geom. Funct. Anal.}, 10(5):1071--1089, 2000.

\bibitem[Bro87]{Brown87}
Kenneth~S. Brown.
\newblock Trees, valuations, and the {B}ieri-{N}eumann-{S}trebel invariant.
\newblock {\em Invent. Math.}, 90(3):479--504, 1987.

\bibitem[CL95]{CohenLustig1995}
Marshall~M. Cohen and Martin Lustig.
\newblock Very small group actions on {${\bf R}$}-trees and {D}ehn twist
  automorphisms.
\newblock {\em Topology}, 34(3):575--617, 1995.

\bibitem[CL99]{CohenLustig1999}
Marshall~M. Cohen and Martin Lustig.
\newblock The conjugacy problem for {D}ehn twist automorphisms of free groups.
\newblock {\em Comment. Math. Helv.}, 74(2):179--200, 1999.

\bibitem[CL16]{CashenLevitt2016}
Christopher~H. Cashen and Gilbert Levitt.
\newblock Mapping tori of free group automorphisms, and the
  {B}ieri-{N}eumann-{S}trebel invariant of graphs of groups.
\newblock {\em J. Group Theory}, 19(2):191--216, 2016.

\bibitem[DT06]{DT}
Nathan~M. Dunfield and Dylan~P. Thurston.
\newblock A random tunnel number one 3-manifold does not fiber over the circle.
\newblock {\em Geom. Topol.}, 10:2431--2499, 2006.

\bibitem[Gro93]{Gromov1993}
M.~Gromov.
\newblock Asymptotic invariants of infinite groups.
\newblock In {\em Geometric group theory, {V}ol. 2 ({S}ussex, 1991)}, volume
  182 of {\em London Math. Soc. Lecture Note Ser.}, pages 1--295. Cambridge
  Univ. Press, Cambridge, 1993.

\bibitem[Hag19]{Hagen2019}
Mark Hagen.
\newblock A remark on thickness of free-by-cyclic groups.
\newblock {\em Illinois J. Math.}, 63(4):633--643, 2019.

\bibitem[Hal49]{Hall}
Marshall Hall, Jr.
\newblock Coset representations in free groups.
\newblock {\em Trans. Amer. Math. Soc.}, 67:421--432, 1949.

\bibitem[HK22]{HughesKielak2022}
Sam Hughes and Dawid Kielak.
\newblock Profinite rigidity of fibring, 2022.
\newblock arXiv:2206.11347 [math.GR].

\bibitem[Kie20]{Kielak}
Dawid Kielak.
\newblock Residually finite rationally solvable groups and virtual fibring.
\newblock {\em J. Amer. Math. Soc.}, 33(2):451--486, 2020.

\bibitem[KKW22]{KKW}
Dawid Kielak, Robert Kropholler, and Gareth Wilkes.
\newblock {$\ell^2$}-{B}etti numbers and coherence of random groups.
\newblock {\em J. Lond. Math. Soc. (2)}, 106(1):425--445, 2022.

\bibitem[Lev09]{Levitt}
Gilbert Levitt.
\newblock Counting growth types of automorphisms of free groups.
\newblock {\em Geom. Funct. Anal.}, 19(4):1119--1146, 2009.

\bibitem[LNW99]{LNW}
Ian~J. Leary, Graham~A. Niblo, and Daniel~T. Wise.
\newblock Some free-by-cyclic groups.
\newblock In {\em Groups {S}t. {A}ndrews 1997 in {B}ath, {II}}, volume 261 of
  {\em London Math. Soc. Lecture Note Ser.}, pages 512--516. Cambridge Univ.
  Press, Cambridge, 1999.

\bibitem[Mac02]{Macura2002}
Nata\v{s}a Macura.
\newblock Detour functions and quasi-isometries.
\newblock {\em Q. J. Math.}, 53(2):207--239, 2002.

\bibitem[NW01]{NW}
Graham~A. Niblo and Daniel~T. Wise.
\newblock Subgroup separability, knot groups and graph manifolds.
\newblock {\em Proc. Amer. Math. Soc.}, 129(3):685--693, 2001.

\bibitem[Sch14]{Schreve14}
Kevin Schreve.
\newblock The strong {A}tiyah conjecture for virtually cocompact special
  groups.
\newblock {\em Math. Ann.}, 359(3-4):629--636, 2014.

\bibitem[Sik87]{Sikorav}
J.-C. Sikorav.
\newblock {\em Homologie de Novikov associ\'{e}e \`{a} une classe de
  cohomologie r\'{e}elle de degr\'{e} un}.
\newblock PhD thesis, Universit\'{e} Paris-Sud, 1987.

\bibitem[Wil19]{Wilkes19}
Gareth Wilkes.
\newblock Virtually abelian quotients of random groups, 2019.
\newblock arXiv:1902.02152 [math.GR].

\bibitem[Wis00]{Wise}
Daniel~T. Wise.
\newblock Subgroup separability of graphs of free groups with cyclic edge
  groups.
\newblock {\em Q. J. Math.}, 51(1):107--129, 2000.

\bibitem[Wis04]{WiseCubes}
D.~T. Wise.
\newblock Cubulating small cancellation groups.
\newblock {\em Geom. Funct. Anal.}, 14(1):150--214, 2004.

\end{thebibliography}

\end{document}